\documentclass[11pt]{amsart}
\usepackage{amssymb,latexsym}
\usepackage{enumerate}
\usepackage[T1]{fontenc}
\usepackage{mathrsfs}
\usepackage{amsmath}
\usepackage{yhmath}
\makeatletter
\@namedef{subjclassname@2010}{%
  \textup{2010} Mathematics Subject Classification}
\makeatother
\makeindex

\usepackage{chngcntr}

\usepackage{mathtools}
\counterwithin{equation}{section}

\newtheorem{theorem}{Theorem}[subsection]

\theoremstyle{definition}

\newtheorem{sit}[theorem]{Situation}
\newtheorem{rem}[theorem]{Remark}

\numberwithin{theorem}{section}

\newcommand{\mb}{\mathbb}
\newcommand{\mc}{\mathcal}
\newcommand{\mf}{\mathfrak}

\newcommand{\s}{\subset}

\begin{document}
\title[Parametrix for the localization of the Bergman metric...]{Parametrix for the localization of the Bergman metric on strictly pseudoconvex domains}
\author{Arkadiusz Lewandowski}
\address{Institute of Mathematics\\ Faculty of Mathematics and Computer Science\\ Jagiellonian University\\ {\L}ojasiewicza 6,
30-348 Kraków, Poland}
\email{Arkadiusz.Lewandowski@im.uj.edu.pl}
\begin{abstract}
We give the parameter version of localization theorem for Bergman metric near the boundary points of strictly pseudoconvex domains. The approximation theorem for square integrable holomorphic functions on such domains in the spirit of Graham-Kerzman is proved in the hereby paper, as well.
\end{abstract}



\maketitle
\section{Introduction} Let $G\s\mb{C}^n$ be a bounded domain. Let $L^2_h(G)$ denote the Hilbert space of all square integrable holomorphic functions on $G$ equipped with the norm $\|\cdot\|_{L^2(G)}$ arising from the standard scalar product $(f,g):=\int_Gf\bar{g}d\mc{L}^{2n}$. It is well known that the Bergman kernel of $G$ restricted to the diagonal may be represented as
$$
\mf{K}_G(z)=\sup\{|f(z)|^2:f\in L^2_h(G):\|f\|_{L^2_h}\leq 1\}.
$$
The Levi form of smooth strictly plurisubharmonic function $\log\mf{K}_G$, denoted by $\mc{L}_{\mf{K}_G}$, may be used to define the Bergman metric of $G$  
$$\displaystyle{
\boldsymbol{\beta}_G(z;X):=\sqrt{\mc{L}_{\mf{K}_G}(z;X)},\quad z\in G, X\in\mb{C}^n.}
$$
This latter function allows the following description
$$
\boldsymbol{\beta}_G(z;X)=\frac{\sup\{|f'_X(z)|:f\in L^2_h(G):\|f\|_{L^2(G)}\leq 1,f(z)=0\}}{\sqrt{\mf{K}_G(z)}}=:\frac{M_G(z;X)}{\sqrt{\mf{K}_G(z)}}
$$
for $z\in G,X=(X_1,\ldots,X_n)\in\mb{C}^n,$ where $f'_X(z):=\sum_{j=1}^n\frac{\partial f}{\partial z_j}X_j.$ 

The following localization result for the Bergman metric has been announced in [\ref{JP}, Theorem 19.3.6]:
\begin{theorem} Let $G$ be a strongly pseudoconvex domain with $\mc{C}^2$ boundary and let $R>0$ be such that for any $\zeta_0\in\partial G$ the set $G\cap\mb{B}(\zeta_0,R)$ is connected. Then, for every $\varepsilon>0$ there exists a $\delta\in(0,R)$ such that 
\begin{enumerate}[\indent \upshape (1')]
\item $M_G(z;X)\leq M_{G\cap\mb{B}(\zeta_0,R)}(z;X)\leq (1+\varepsilon)M_G(z;X)$
\item $\mf{K}_G(z)\leq\mf{K}_{G\cap\mb{B}(\zeta_0,R)}(z)\leq(1+\varepsilon)\mf{K}_G(z)$
\item $(1+\varepsilon)^{-1}\boldsymbol{\beta}_{G\cap\mb{B}(\zeta_0,R)}(z;X)\leq\boldsymbol{\beta}_G(z;X)\leq\sqrt{1+\varepsilon}\boldsymbol{\beta}_{G\cap\mb{B}(\zeta_0,R)}(z;X)$
\end{enumerate}
for all $z\in G\cap\mb{B}(\zeta_0,\delta),X\in\mb{C}^n.$ Additionally, $R$ can be chosen so that $\delta$ does not depend on the boundary point $\zeta_0.$
\end{theorem} 
\noindent See also [\ref{Die}, Theorem 1] for another result of this type. We would like to give a parameter version of this theorem, with uniform size of respective neighborhoods of boundary points. In the beginning, let us settle the following:
\begin{sit}
Let $(G_t)_{t\in T}$ be a family of bounded strictly pseudoconvex domains with $\mc{C}^2$-smooth boundaries, where $T$ is a compact metric space with associated metric $d$. Suppose we have a domain $U\s\s\mb{C}^n$ such that
\begin{enumerate}[(i)]
\item $\displaystyle{\bigcup_{t\in T}\partial G_t\s\s U},$
\item for each $t\in T$ there exists a defining function $r_t\in\mc{C}^2(U)$ for $G_t$ such that its Levi form $\mc{L}_{r_t}$ is positive on $U\times(\mb{C}^n\setminus\{0\}),$
\item for any $\varepsilon>0$ there exists a $\delta>0$ such that for any $s,t\in T$ with $d(s,t)\leq \delta$ there is $\|r_t-r_s\|_{\mc{C}^2(U)}<\varepsilon$.\label{C2}
\end{enumerate}\label{Situation}
\end{sit}
We will prove the following
\begin{theorem}\label{Bergman}
Let $(G_t)_{t\in T}$ be a family of strictly pseudoconvex domains as in Situation \ref{Situation}. Then there exists an $R>0$ such that for any $t\in T$ and $\zeta_t\in\partial G_t$ the set $G_t\cap\mb{B}(\zeta_t,R)$ is connected, and for any such $R$, and any $\varepsilon>0$ there exists a $\theta\in(0,R)$ such that for any $t\in T$ and $\zeta_t\in\partial G_t$ we have
\begin{enumerate}[\indent \upshape (1)]
\item $M_{G_t}(z;X)\leq M_{G_t\cap\mb{B}(\zeta_t,R)}(z;X)\leq (1+\varepsilon)M_{G_t}(z;X)$
\item $\mf{K}_{G_t}(z)\leq\mf{K}_{G_t\cap\mb{B}(\zeta_t,R)}(z)\leq(1+\varepsilon)\mf{K}_{G_t}(z)$
\item $(1+\varepsilon)^{-1}\boldsymbol{\beta}_{G_t\cap\mb{B}(\zeta_t,R)}(z;X)\leq\boldsymbol{\beta}_{G_t}(z;X)\leq\sqrt{1+\varepsilon}\boldsymbol{\beta}_{G_t\cap\mb{B}(\zeta_t,R)}(z;X)$
\end{enumerate}
for all $z\in G_t\cap\mb{B}(\zeta_t,\theta),X\in\mb{C}^n.$ 
\end{theorem}
\begin{rem} As we have indicated before, all the above estimates are known for single $\mc{C}^2$-smooth strictly pseudoconvex domain $G.$ The novelty of our result is the fact that for the family of strictly pseudoconvex domains as in Situation \ref{Situation} the respective neighbourhoods of boundary points may be chosen to be of uniform size, independently of $t\in T$ and $\zeta_t\in \partial G_t.$
\end{rem}
The main ingredient in the proof of Theorem \ref{Bergman} will be the following 
\begin{theorem}\label{L2h}
Let $(G_t)_{t\in T}$ be a family of strictly pseudoconvex domains as in Situation \ref{Situation}. Then there exist an  $R>0$ such that the set $G_t\cap\mb{B}(\zeta,R)$ is connected for any $t\in T,\zeta\in\partial G_t$ and for every such $R$ there exists a $\rho< R$ with the property that for any $\varepsilon>0$ there exists an $L=L(\varepsilon, R)>0$ such that for any $t\in T,\zeta_t\in\partial G_t, f_t\in L^2_h(G_t\cap\mb{B}(\zeta_t,R)),$ and any point $w\in G_t\cap\mb{B}(\zeta_t,\rho)$ there exist an $\hat{f}_t\in L^2_h(G_t)$ such that
\begin{enumerate}[\indent\upshape(A)]
\item $D^{\alpha}\hat{f}_t(w)=D^{\alpha}f_t(w)$ for $|\alpha|\leq 1$,
\item $\|\hat{f}_t\|_{L^2(G_t)}\leq L\|f_t\|_{L^2(G_t\cap\mb{B}(\zeta_t,R))}$, 
\item $\|\hat{f}_t-f_t\|_{L^2(G_t\cap\mb{B}(\zeta_t,\rho))}<\varepsilon\|f\|_{L^2(G_t\cap\mb{B}(\zeta_t,R))}.$
\end{enumerate}
\end{theorem}
\begin{rem} 
For a fixed domain $G=G_{t_0}$, Theorem \ref{L2h} may be viewed as a variant of Theorem 1 from [\ref{Die}]. In that result, one starts with a square integrable holomorphic function defined on some small open set touching fixed boundary point that sticks out of $G$, and gets the approximation on some smaller subset of $G$ by functions from $L^2_h(G)$. Here, we start with $L^2_h$ function on small subset of $G$ touching $\partial G$, and get the same type of approximation. Note that our proof of Theorem \ref{L2h} is essentially different than the proof of the mentioned result (see Section \ref{proofs} for the details). 
\end{rem}
\begin{rem}
In [\ref{AL2}] we proved an analogous result for bounded holomorphic functions, instead of square integrable holomorphic ones (see also [\ref{Gra}]). The proof of said assertion is based on stating and solving certain family of subtle $\bar{\partial}$ problems on some deformations of the domains $G_t$, with estimates that do not depend on $t\in T$ and $\zeta_t\in\partial G_t.$ The idea of the proof of Theorem \ref{L2h} is similar, although here some estimations must be carried out more carefully. The details are given in Section \ref{proofs}. The issue of the domain dependence for the $\bar{\partial}$ equation has been recently investigated in [\ref{GK}].
\end{rem}
In Section \ref{preliminaries} we collect some preliminary facts about the strictly pseudoconvex domains. The proofs of Theorems \ref{Bergman} and \ref{L2h} are given in Section ~\ref{proofs}.

\section{Strictly pseudoconvex domains}\label{preliminaries}

A bounded domain $G\s\mb{C}^n$ is called \emph{strictly pseudoconvex} if there exist a neighborhood $U$ of $\partial G$ and a \emph{defining function} $r\in\mc{C}^2(U,\mb{R})$ such that
\begin{enumerate}
{\item[(I)] $G\cap U=\{z\in U:r(z)<0\}$,\label{Condition1}}
{\item[(II)] $(\mb{C}^n\setminus\overline{G})\cap U=\{z\in U:r(z)>0\},\label{Condition2}$}
{\item[(III)] $\nabla r(z)\neq 0$ for $z\in \partial G,$ where $\nabla r(z):=\left(\frac{\partial r}{\partial\overline{z}_1}(z),\cdots,\frac{\partial r}{\partial\overline{z}_n}(z)\right)$,\label{Condition3}}
\end{enumerate}
and with the property that $$\mc{L}_r(z;X)>0\text{\ for\ } z\in\partial G\text{\ and\ nonzero\ }X\in T_z^{\mb{C}}(\partial D),$$  
where $\mc{L}_r$ denotes the Levi form of $r$ and $T_z^{\mb{C}}(\partial G)$ is the complex tangent space to $\partial G$ at $z$.\\ 
\indent It is known that $U$ and $r$ can be chosen to satisfy (I)-(III) and, additionally:
\begin{enumerate}
{\item[(IV)] $\mc{L}_r(z;X)>0$ for $z\in U$ and all nonzero $X\in\mb{C}^n,$\label{Condition4}}
\end{enumerate}
cf. [\ref{Kra1}]. It is well known that every boundary point of such $G$ is a \emph{peak point} with respect to $\mc{O}(\overline{G})$ - the family of all functions holomorphic in a neighbourhood of the closure of $G$, i.e., for any $\zeta\in\partial G$ there exist a function $f\in\mc{O}(\overline{G})$ such that $f(\zeta)=1$ and $|f(z)|<1$ for $z\in G\setminus\{\zeta\}.$ Actually, much more is known: for example, in [\ref{AL1}] we have proved the following 
\begin{theorem}
Let $(G_t)_{t\in T}$ be a family of strictly pseudoconvex domains as in Situation \ref{Situation}.
Then there exists an $\varepsilon>0$ such that for any $\eta_1<\varepsilon$ there exist an $\eta_2>0$ and positive constants $d_1,d_2$ such that for any $t\in T$ there exist a domain $\widehat{G_t}$ containing $\overline{G_t}$, and functions $h_t(\cdot;\zeta)\in\mc{O}(\widehat{G_t}),\zeta\in\partial G_t$ fulfilling the following conditions:
\begin{enumerate}
{\item[\emph{(a)}] $h_t(\zeta;\zeta)=1, |h_t(\cdot;\zeta)|<1$ on $\overline{G_t}\setminus\{\zeta\}$ (in particular, $h_t(\cdot;\zeta)$ is a peak function for $G_t$ at $\zeta$), \label{a}}
{\item[\emph{(b)}] $|1-h_t(z;\zeta)|\leq d_1\|z-\zeta\|, z\in\widehat{G_t}\cap\mb{B}(\zeta,\eta_2),$\label{b}}
{\item[\emph{(c)}] $|h_t(z;\zeta)|\leq d_2<1, z\in\overline{G_t},\|z-\zeta\|\geq\eta_1.$\label{c}}
\end{enumerate}
\label{Main}
\end{theorem}
\begin{rem}
We would like to draw the Reader's attention to the fact that all the constants $\varepsilon,\eta_2,d_1,d_2$ in Theorem \ref{Main} may be chosen independently of $t$, which is of great importance in the proofs of Theorems \ref{Bergman} and \ref{L2h}.  
\end{rem}
\section{The proofs}\label{proofs}
We begin with the proof of Theorem \ref{L2h} on approximation of $L^2_h$ functions.
\begin{proof}[Proof of Theorem \ref{L2h}] The proof of Theorem \ref{L2h} is similar to the proof of Theorem 1.5 from [\ref{AL2}]. Therefore, we here focus mainly on the part of it that is new.\\   
Set $\eta_2<\eta_1,d_1,d_2<1,\widehat{G_t},$ and $h_t(\cdot;\zeta)$ for $t\in T,\zeta\in\partial G_t$ according to Theorem \ref{Main}, where $\eta_1$ is small enough to assure that the set $G_t\cap\mb{B}(\zeta,R)$ is connected for every $t\in T$ and $\zeta\in\partial G_t,$ where $R:=2\eta_1.$ Replacing $h_t$ with $\frac{h_t+3}{4}$ we may assume that $|h_t(z;\zeta)|\geq\frac{1}{2}, z\in\overline{G_t},\zeta\in\partial G_t.$\\
Let $d_3\in(d_2,1)$ and choose $0<\eta\leq\frac{\eta_2}{2}$ such that for any $t\in T$ we have $\mb{B}(\zeta;2\eta)\s\widehat{G_t}$ for all $\zeta\in\partial G_t$ as well as $|h_t(z;\zeta)|\geq d_3$ whenever $\zeta\in\partial G_t$ and $\|z-\zeta\|\leq\eta$ (this is possible because of the uniform choice of $d_1$ in Theorem \ref{Main}). Define $\rho:=\min\{\frac{\eta}{2},\frac{\eta_1}{5}\}.$\\ 
As in [\ref{AL2}], we show that for any $s\in T$ we may choose points $\zeta_1^s,\ldots,\zeta_{N_s}^s\in\partial G_s$ such that $\displaystyle{\partial G_s\s\bigcup_{j=1}^{N_s}\mb{B}(\zeta_j^s,\rho)},$ and with the property that for any $j\in\{1,\ldots,N_s\}$ we can find strictly pseudoconvex $\mc{C}^2$ deformation $G_j^s$ of $G_s$ near $\zeta_j^s$ such that
\begin{equation} G_s\s G_j^s\s\widehat{G_s}\cap G_s^{(\eta)} (G_s^{(\eta)} \text{\rm\ standing\  for\  the\ }\eta\text{-hull\ of\ }G_s),
\end{equation}
\begin{equation} 
\overline{G_s}\cap\overline{\mb{B}(\zeta_j^s,2\rho)}\s\s G_j^s \text{\ and\ dist}(\overline{G_s}\cap\overline{\mb{B}(\zeta_j^s,2\rho)},\partial G_j^t)\geq {\beta}>0,\label{eq02}
\end{equation}
and
\begin{equation}
G_s\setminus\mb{B}(\zeta_j^s,4\rho)=G_j^s\setminus\mb{B}(\zeta_j^s,4\rho),
\end{equation}
%\item The estimate $C$ for the solution of $\bar{\partial}$-problem from Remark \ref{R1} is good for $G_j^s$, 
where the constant $\beta$ do not depend on $s$. Moreover, by H\"ormander's $\bar{\partial}$~theory, things may be settled so that there is a positive constant $C$ with the property that for any $s\in T$, and any $j\in\{1,\ldots, N_s\}$ the following estimate holds: for any $\bar{\partial}$-closed $(0,1)$-form $\alpha\in L^2_{(0,1)}(G_j^s)\cap\mc{C}^{\infty}(G_j^s)$ there exists a function $v\in L^2(G_j^s)\cap\mc{C}^{\infty}(G_j^s)$ with $\bar{\partial}v=\alpha$ and such that $\|v\|_{L^2(G_j^s)}\leq C\|\alpha\|_{L_{(0,1)}^2(G_j^s)},$ cf. [\ref{H1}].\footnote{Note that if in Situation \ref{Situation} we assumed all the domains are of class $\mc{C}^4$ and vary in $\mc{C}^4$-topology on domains, the existence of such $C$ would follow also from [\ref{Ker1}, Theorem 1.2.1]~ and remarks following it, together with the compactness of $T$. Eventually, if things were of class $\mc{C}^3$, one could use [\ref{Ran}, Theorem VII. 5.6 and Corollary VII.5.9].}\\
Fix now $t=t_0\in T$ and $\zeta_0\in\partial G_t.$ Let $f\in L^2_h(G_t\cap\mb{B}(\zeta_0,R))$ and take a point $w\in G_t\cap\mb{B}(\zeta_0,\rho)$.\\
There exists a $j_0\in\{1,\ldots N_t\}$ such that $\zeta_0\in\mb{B}(\zeta_{j_0}^t,\rho)$. To simplify the notation, let us assume without loss of generality that $j_0=1.$\\
Choose a $\chi\in\mc{C}^{\infty}(\mb{C}^n,[0,1])$ such that $\chi\equiv 1$ on $\mb{B}(\zeta_0,\frac{6\eta_1}{5})$ and $\chi\equiv 0$ outside $\mb{B}(\zeta_0,\frac{9\eta_1}{5})$ and define $\alpha_t:=(\bar{\partial}\chi)f$ on $G_t\cap\mb{B}(\zeta_0,R)=G_t\cap\mb{B}(\zeta_0,2\eta_1)$ and $\alpha_t:=0$ on $G_t\setminus\mb{B}(\zeta_0,2\eta_1)$. Note that in view of the fact that $\alpha_t\equiv 0$ on $(G_t\cap\mb{B}(\zeta_0,\frac{6\eta_1}{5}))\cup(G_t\setminus\mb{B}(\zeta_0,\frac{9\eta_1}{5}))$, after trivial extension by zero, it can be treated as a $\bar{\partial}$-closed $(0,1)$-form of class $\mc{C}^{\infty}\cap L^2$ on $G_1^t$. Similarly, for $k\in\mb{N}$ the form $\widetilde{\alpha_t}=\widetilde{\alpha_{t,k}}:=h_t(\cdot;\zeta_0)^k\alpha_t$ is of class $\mc{C}^{\infty}\cap L^2$ on $G_1^t$. Let us now consider the equation
\begin{equation}
\bar{\partial}v^t_k=\widetilde{\alpha_t}.\label{E1}
\end{equation}
There exists a solution $v_k^t\in\mc{C}^{\infty}(G_1^t)\cap L^2(G_1^t)$ of the problem (\ref{E1}) such that
$$
\|v_k^t\|_{L^2(G_1^t)}\leq C\|\widetilde{\alpha_t}\|_{L^2_{(0,1)}G_1^t},
$$
with, as mentioned above, $C$ independent of $t\in T$ and of the choice of $\zeta_0.$\\
Further, we have
\begin{multline*}
\|v_k^t\|_{L^2(G_1^t)}\leq C\|\widetilde{\alpha_t}\|_{L^2_{(0,1)}G_1^t}\\=C\sqrt{\int_{G_t\cap(\mb{B}(\zeta_0,\frac{9\eta_1}{5})\setminus\mb{B}(\zeta_0,\frac{6\eta_1}{5}))}|h_t(z;\zeta_0)|^{2k}|f|^2\sum_{j=1}^n\big|\frac{\partial \chi}{\partial z_j}(z)\big|^2d\mc{L}^{2n}(z)}\\\leq\widetilde{C}d_2^k\|f\|_{L^2(G_t\cap\mb{B}(\zeta_0,R))},
\end{multline*}
where the constant $\widetilde{C}$ depends only on $\eta_1$ (in particular, it does not depend on $t$).\\
Define the function $f_k:=\chi f-h_t(\cdot;\zeta_0)^{-k}v_k^t$ and observe it is holomorphic on $G_t$. Consequently, the function $h_t(\cdot;\zeta_0)^{-k}v_k^t$ is holomorphic on $G_1^t\cap\mb{B}(\zeta_0,\eta)$.\\
Furthermore
$$
\|h_t(\cdot;\zeta_0)^{-k}v_k^t\|_{L^2(G_1^t\cap\mb{B}(\zeta_0,\eta))}\leq \widetilde{C}\left(\frac{d_2}{d_3}\right)^k\|f\|_{L^2(G_t\cap\mb{B}(\zeta_0,R))}.
$$
As in [\ref{AL2}], observe that ${G_t}\cap\mb{B}(\zeta_0,\rho)\s\s G_1^t\cap\mb{B}(\zeta_0,\eta)$ and, by (\ref{eq02}), the distance of the former set to the boundary of the latter one is bounded from below by some positive constant, independent of $t$ and the choice of $\zeta_0.$\\ 
We therefore have
$$
\|f_k-f\|_{L^2(G_t\cap\mb{B}(\zeta_0,\eta))}\leq\|h_t(\cdot;\zeta_0)^{-k}v_k^t\|_{L^2(G_1^t\cap\mb{B}(\zeta_0,\eta))}\leq \widetilde{C}\left(\frac{d_2}{d_3}\right)^k\|f\|_{L^2(G_t\cap\mb{B}(\zeta_0,R))}.
$$
Also,
$$
\|f_k-f\|_{G_t\cap\mb{B}(\zeta_0,\rho)}\leq\widehat{C}\|h_t(\cdot;\zeta_0)^{-k}v_k^t\|_{L^2(G_1^t\cap\mb{B}(\zeta_0,\eta))}\leq \widehat{C}\widetilde{C}\left(\frac{d_2}{d_3}\right)^k\|f\|_{L^2(G_t\cap\mb{B}(\zeta_0,R))},
$$
where the positive constant $\widehat{C}$ does not depend on $t$ and $\zeta_0$, in virtue of Bergman inequality and the choice of $\beta$.\\
Similarly, if we take $\kappa>0$ so small that $(G_t\cap\mb{B}(\zeta_0,\rho))^{(\kappa)}$ - the $\kappa$-hull of $G_t\cap\mb{B}(\zeta_0,\rho)$ - stays compactly in $G_1^t\cap\mb{B}(\zeta_0,\eta)$ with the distance to the boundary independent of $t$ and $\zeta_0$, using the Cauchy inequality, we get the estimate
\begin{multline*}
\left\|\frac{\partial f_k}{\partial z_j}-\frac{\partial f}{\partial z_j}\right\|_{G_t\cap\mb{B}(\zeta_0,\rho)}=\left\|\frac{\partial}{\partial z_j}(h_t(\cdot;\zeta_0)^{-k}v_k^t)\right\|_{G_t\cap\mb{B}(\zeta_0,\rho)}\\\leq 
C'\|h_t(\cdot;\zeta_0)^{-k}v_k^t\|_{G_t\cap\mb{B}(\zeta_0,\rho))^{(\kappa)}}\leq C''\|h_t(\cdot;\zeta_0)^{-k}v_k^t\|_{L^2(G_1^t\cap\mb{B}(\zeta_0,\eta))}\\\leq C''\widetilde{C}\left(\frac{d_2}{d_3}\right)^k\|f\|_{L^2(G_t\cap\mb{B}(\zeta_0,R))},
\end{multline*}
where the positive constants $C',C''$ may be chosen independently of $t$ and $\zeta_0$.\\
Fix $\varepsilon>0.$ Define $\hat{f}_k\in\mc{O}(G_t)$ by $\hat{f}_k(z):=f_k(z)+p(z),$ where
$$
p(z):=f(w)-f_k(w)+\sum_{j=1}^n\left(\frac{\partial f}{\partial z_j}(w)-\frac{\partial f_k}{\partial z_j}(w)\right)(z_j-w_j).
$$
It can be easily checked that $\hat{f}_k(w)=f(w),$ as well as $\frac{\partial \hat{f}_k}{\partial z_j}(w)=\frac{\partial f}{\partial z_j}(w).$ Furthermore,
\begin{multline*}
\|\hat{f}_k-f\|_{L^2(G_t\cap\mb{B}(\zeta_0,\rho))}\leq \|f_k-f\|_{L^2(G_t\cap\mb{B}(\zeta_0,\rho))}+\sqrt{\mc{L}^{2n}(U)}|f(w)-f_k(w)|\\+\sqrt{\mc{L}^{2n}(U)}{\text{diam}U}\sum_{j=1}^n\left|\frac{\partial f}{\partial z_j}(w)-\frac{\partial f_k}{\partial z_j}(w)\right|\\\leq \widetilde{C}(1+\sqrt{\mc{L}^{2n}U}(\widehat{C}+{\text{diam}U}C''))\left(\frac{d_2}{d_3}\right)^k\|f\|_{L^2(G_t\cap\mb{B}(\zeta_0,R))}\\\leq\varepsilon\|f\|_{L^2(G_t\cap\mb{B}(\zeta_0,R))},
\end{multline*}
provided that $k=k_0$ is large enough (observe it depends only on $\varepsilon$ and $\eta_1$). Define $\hat{f}:=\hat{f}_{k_0}.$ We estimate:
\begin{multline*}
\|\hat{f}\|_{L^2(G_t)}\leq\|f_{k_0}\|_{L^2(G_t)}+\sqrt{\mc{L}^{2n}(U)}|f(w)-f_{k_0}(w)|\\+\sqrt{\mc{L}^{2n}(U)}{\text{diam}U}\sum_{j=1}^n\left|\frac{\partial f}{\partial z_j}(w)-\frac{\partial f_{k_0}}{\partial z_j}(w)\right|\\\leq
\|f\|_{L^2(G_t\cap\mb{B}(\zeta_0,R))}+\|h_t(\cdot;\zeta_0)^{-k_0}v_{k_0}^t\|_{L^2(G_t)}+\sqrt{\mc{L}^{2n}(U)}|f(w)-f_{k_0}(w)|\\+\sqrt{\mc{L}^{2n}(U)}{\text{diam}U}\sum_{j=1}^n\left|\frac{\partial f}{\partial z_j}(w)-\frac{\partial f_{k_0}}{\partial z_j}(w)\right|\\\leq (1+\widetilde{C}(2d_2)^{k_0}+\varepsilon)\|f\|_{L^2(G_t\cap\mb{B}(\zeta_0,R))}=:L\|f\|_{L^2(G_t\cap\mb{B}(\zeta_0,R))},
\end{multline*}
which concludes the proof of Theorem \ref{L2h}. 
\end{proof}
Let us pass to the proof of localization result for the Bergman metric. It is based on similar idea as the proof of Satz 1 from [\ref{Die}]. We are able to get, however, more information about the respective neighbourhoods of boundary points for the localization, and the independence of their size of the parameter $t$.
\begin{proof}[Proof of Theorem \ref{Bergman}] Let $R$ as in Theorem \ref{L2h}. Observe that the first inequalities in (1) and (2) are obvious. Also, (3) is an easy consequence of (1) and (2). We shall first prove the second inequality in (2). Set $\varepsilon>0$ and $\varepsilon'>0$ such that $(1+2\varepsilon')^2< 1+\frac{\varepsilon}{2}.$ Fix $t$ and $\zeta_0\in\partial G_t$. Let $f\in L^2_h(G_t\cap\mb{B}(\zeta_0,R))$ be such that $\|f\|_{L^2(G_t\cap\mb{B}(\zeta_0,R))}\leq 1$ and let $z\in G_t\cap\mb{B}(\zeta_0,\rho).$ From Theorem \ref{L2h} it follows that there exists an $\hat{f}\in L^2_h(G_t)$ such that $\hat{f}(z)=f(z)$ as well as $\|\hat{f}\|_{L^2(G_t)}\leq L$ and $\|\hat{f}-f\|_{L^2(G_t\cap\mb{B}(\zeta_0,\rho))}<\varepsilon',$ and the constant $L$ depends only on $\varepsilon'$ and $R$.\\
Fix a $\theta'<\rho$ and let $h_t(\cdot;\zeta_0)$ be a function given by Theorem \ref{Main} with $\eta_1=\theta'$ and  unrestricted other parameters. Take $1>\gamma>0$ such that $\gamma L<\varepsilon'$ and small enough to ensure the inequality $\frac{1+\frac{\varepsilon}{2}}{1+\varepsilon}\leq (1-\gamma)^2.$ Also, because of the choice of $\theta'$ and in virtue of Theorem \ref{Main}, there exists a $k\in\mb{N}$, independent on $t,\zeta_0,$ and $h_t(\cdot;\zeta_0)$ such that $|h_t(\cdot;\zeta_0)^k|\leq\gamma$ on $G_t\setminus\mb{B}(\zeta_0,\theta')$. Finally, by (b) in Theorem \ref{Main}, we may choose a $\theta<\theta'$, independent of $t,\zeta_0,$ and $h_t(\cdot;\zeta_0)$ such that $|h_t(\cdot;\zeta_0)^k|>1-\gamma$ on $G_t\cap\mb{B}(\zeta_0,\theta).$\\ 
Consider the function $\tilde{f}:=h_t(\cdot;\zeta_0)^k\hat{f}$. We have
\begin{multline*}
\|\tilde{f}\|_{L^2(G_t)}=\|h_t(\cdot;\zeta_0)^k\hat{f}\|_{L^2(G_t)}\leq \|h_t(\cdot;\zeta_0)^k\hat{f}\|_{L^2(G_t\setminus\mb{B}(\zeta_0,\theta'))}\\+\|h_t(\cdot;\zeta_0)^k\hat{f}\|_{L^2(G_t\cap\mb{B}(\zeta_0,\theta'))}\leq \gamma L+\|\hat{f}\|_{L^2(G_t\cap\mb{B}(\zeta_0,\theta'))}\leq \gamma L+\varepsilon' +1\leq 1+2\varepsilon'.
\end{multline*}
Furthermore, if from the beginning we have had $z\in G_t\cap\mb{B}(\zeta_0,\theta)$, then
$$
|\tilde{f}(z)|=|h_t(z;\zeta_0)^k||\hat{f}(z)|=|h_t(z;\zeta_0)^k||{f}(z)|\geq(1-\gamma)|f(z)|.
$$
This implies that for $z\in G_t\cap\mb{B}(\zeta_0,\theta)$ the following estimates hold
$$
\mf{K}_{G_t}(z)\geq\frac{|\tilde{f}(z)|^2}{(1+2\varepsilon')^2}\geq\frac{(1-\gamma)^2|f(z)|^2}{(1+2\varepsilon')^2},
$$
and consequently, because of arbitrariness of $f$,
$$
(1+\varepsilon)\mf{K}_{G_t}(z)\geq\frac{(1+2\varepsilon')^2}{(1-\gamma)^2}\mf{K}_{G_t}(z)\geq\mf{K}_{G_t\cap\mb{B}(\zeta_0,R)}(z).
$$
This establishes (2). For the proof of (1) we only need to notice that if $f$ in the proof of (2) were chosen so that $f(z)=0$, then $\tilde{f}(z)=0$, for $z\in G_t\cap\mb{B}(\zeta_0,\theta).$ Then for such $z$ and any $X\in\mb{C}^n$ we have
\begin{multline*}
M_{G_t}(z;X)\geq\frac{|\tilde{f}'_X(z)|}{1+2\varepsilon'}=\frac{|(h_t(\cdot;\zeta_0)^k\hat{f})'_X(z)|}{1+2\varepsilon'}=\frac{|(h_t(\cdot;\zeta_0)^k{f})'_X(z)|}{1+2\varepsilon'}\\\geq\frac{(1-\gamma)|{f}'_X(z)|}{1+2\varepsilon'},
\end{multline*}
which yields
$$
(1+\varepsilon)M_{G_t}(z;X)\geq\frac{1+2\varepsilon'}{1-\gamma}M_{G_t}(z;X)\geq M_{G_t\cap\mb{B}(\zeta_0,R)}(z;X),
$$
and this concludes the proof of Theorem \ref{Bergman}.
\end{proof}

\end{document}